\documentclass[12pt]{amsart}
\usepackage{amssymb,latexsym,placeins,url} 

\newcommand{\Aff}{{\mathbb A}}
\newcommand{\C}{{\mathbb C}}
\newcommand{\coeff}{\operatorname{coeff}}
\newcommand{\Diag}{\operatorname{Diag}}
\newcommand{\disc}{\operatorname{disc}}
\newcommand{\eps}{\varepsilon}
\newcommand{\F}{{\mathbb F}}
\newcommand{\Gal}{\operatorname{Gal}}
\def\H{{\mathcal H}}
\def\I{{\mathcal I}}
\def\Im{{\rm Im}}
\newcommand{\isom}{\cong}
\newcommand{\Jac}{\operatorname{Jac}}
\newcommand{\lcm}{\operatorname{lcm}}
\newcommand{\phihat}{{\widehat{\phi}}}
\newcommand{\PP}{{\mathbb P}}
\newcommand{\PSL}{{\operatorname{PSL}}}
\newcommand{\Q}{{\mathbb Q}}
\newcommand{\Qbar}{{\overline{\Q}}}
\newcommand{\ra}{\longrightarrow}
\newcommand{\SL}{\operatorname{SL}}
\newcommand{\tr}{\operatorname{tr}}
\newcommand{\vv}{{\mathbf v}}
\newcommand{\ww}{{\mathbf w}}
\newcommand{\Z}{{\mathbb Z}}

                      {\hspace*{\fill}\nobreak$\Box$\par\medskip}
                       {\hspace*{\fill}\nobreak$\Box$\par\medskip}

\newtheorem{Proposition}{Proposition}[section]
\newtheorem{Theorem}[Proposition]{Theorem}
\newtheorem{Lemma}[Proposition]{Lemma}

\newtheorem{Conjecture}[Proposition]{Conjecture}

\theoremstyle{definition}

\newtheorem{Definition}[Proposition]{Definition}
\newtheorem{Remark}[Proposition]{Remark}

\newtheorem{Example}[Proposition]{Example}

\begin{document}
\date{3rd June 2021}
\title[On pairs of $17$-congruent elliptic curves]
{On pairs of $17$-congruent elliptic curves}

\author{T.A.~Fisher}
\address{University of Cambridge,
          DPMMS, Centre for Mathematical Sciences,
          Wilberforce Road, Cambridge CB3 0WB, UK}
\email{T.A.Fisher@dpmms.cam.ac.uk}

\renewcommand{\baselinestretch}{1.1}
\renewcommand{\arraystretch}{1.3}

\renewcommand{\theenumi}{\roman{enumi}}

\begin{abstract}
  We compute explicit equations for the surfaces $Z(17,1)$ and
  $Z(17,3)$ parametrising pairs of $17$-congruent elliptic curves. We
  find that each is a double cover of the same elliptic $K3$-surface.
  We use these equations to exhibit the first non-trivial example of a
  pair of symplectically $17$-congruent elliptic curves over the
  rationals. We also compute the corresponding genus $2$ curve whose
  Jacobian has a $(17,17)$-splitting.
\end{abstract}

\maketitle

\section{Introduction}

Let $p$ be a prime number. Elliptic curves over the rationals are said
to be {\em $p$-congruent} if their $p$-torsion subgroups are
isomorphic as Galois modules, and {\em symplectically} $p$-congruent
if the isomorphism can be chosen to respect the Weil pairing.  For
example if $\phi : E \to E'$ is an isogeny of degree $d$, and $d$ is
coprime to $p$, then $E$ and $E'$ are $p$-congruent, and
symplectically $p$-congruent if $d$ is a quadratic residue mod $p$.
Such congruences, arising from an isogeny, are said to be trivial.

Examples of non-trivial symplectic $p$-congruences were previously
known for all primes $p \le 13$. We exhibit the first such example
with $p = 17$. Specifically, the elliptic curves
\begin{align*}
&E_1: & y^2 + x y &= x^3 - x^2 - 128973503459 x + 17827877649739965 \\ 
&E_2: & y^2 + x y &= x^3 - x^2 - 184201215542543714 x - 
          34187608332483214491862380 
\end{align*}
with conductors 
\begin{align*}
N(E_1) &= 279809270 = 2 \cdot 5 \cdot 13 \cdot 59 \cdot 191^2, \\
N(E_2) &= 3077901970 = 2 \cdot 5 \cdot 11 \cdot 13 \cdot 59 \cdot 191^2,
\end{align*}
are symplectically $17$-congruent. This claim may be verified using
either of the techniques we review in Sections~\ref{sec:verify}
and~\ref{sec:red}.

A pair of anti-symplectically $17$-congruent elliptic curves was
previously found by Cremona \cite{Billerey,CF,7and11}. These are the
elliptic curves
\begin{align*} 
&E'_1: & y^2 + x y &= x^3 - 8 x + 27 \\
&E'_2: & y^2 + x y &= x^3 + 8124402 x - 11887136703 
\end{align*}
with conductors
\begin{align*}
N(E'_1) &= 3675 = 3 \cdot 5^2 \cdot 7^2, \\
N(E'_2) &= 47775 = 3 \cdot 5^2 \cdot 7^2 \cdot 13.
\end{align*}

In \cite{congr13} we completed the proof that for all primes $p \le
13$ there are infinitely many non-trivial pairs of $p$-congruent
elliptic curves (with infinitely many pairs of $j$-invariants) both
symplectic and anti-symplectic. The Frey-Mazur conjecture predicts
that for $p$ sufficiently large, there are no such examples.  We
suggest the following strong form of their conjecture.

\begin{Conjecture} 
\label{conj1}
Let $p \ge 17$ be a prime. Then any pair of $p$-congruent elliptic
curves over the rationals is either explained by an isogeny, or the
elliptic curves are simultaneous quadratic twists of one of the pairs
$(E_1,E_2)$ or $(E'_1,E'_2)$.
\end{Conjecture}

We say that a $p$-congruence has {\em power} $k$ if the isomorphism of
$p$-torsion subgroups raises the Weil pairing to the power $k$. Let
$Z(p,k)$ be the surface parametrising all pairs of elliptic curves
that are $p$-congruent with power $k$, up to simultaneous quadratic
twist. This surface comes with an involution $\iota$ whose moduli
interpretation is that we swap over the two elliptic curves.  We write
$W(p,k)$ for the quotient of $Z(p,k)$ by $\iota$.  These surfaces only
depend (up to isomorphism) on whether $k$ is a quadratic residue or a
quadratic non-residue mod $p$.  As above, we call these the {\em
  symplectic} and {\em anti-symplectic} cases.
 
For $p \le 13$ it is known \cite{moduli, congr13, Kumar} that the
surfaces $W(p,k)$ are rational (i.e., birational to $\PP^2$) over
$\Q$.

\begin{Theorem} 
\label{thm1}
The surfaces $W(17,1)$ and $W(17,3)$ are birational over $\Q$ to the
elliptic K3-surface with Weierstrass equation
\begin{equation}
\label{Weqn0}
y^2 + (T + 1)(T - 2) x y + T^3 y = x^3 - x^2.
\end{equation}
The surfaces $Z(17,1)$ and $Z(17,3)$ are birational over $\Q$ to the
double covers $z^2 = F_1(T,x,y)$ and $z^2 = F_3(T,x,y)$ where $F_1$
and $F_3$ are recorded in Appendix~\ref{app:formulae}.
\end{Theorem}

Let $j_1$ and $j_2$ be the rational functions on $Z(p,k)$ giving the
$j$-invariants of the two elliptic curves.  Then $j_1 + j_2$ and $j_1
j_2$ are rational functions on $W(p,k)$. In the cases $(p,k) = (17,1)$
and $(17,3)$ we have also computed these rational functions. The
formulae are too complicated to record here, but are available
electronically from~\cite{magmafile}.

We used these formulae to find the pairs of elliptic curves
$(E_1,E_2)$ and $(E'_1,E'_2)$ specified above. In each case the curves
are not isogenous, since for example they do not have the same
conductor.  The only previous method for finding such examples was to
search in tables of elliptic curves with small conductor. (See for
example \cite[Section 3]{CF} or \cite[Section 4.3]{BM}.) Accordingly
only the second pair was previously known.

Our evidence for Conjecture~\ref{conj1} when $p=17$ is that we
searched for further rational points on $Z(17,1)$ and $Z(17,3)$, but
none of the points we found give rise to new pairs of $17$-congruent
elliptic curves.  Conjecture~\ref{conj1} has also been verified by
Cremona and Freitas~\cite[Theorem 1.3 and Section 3.7]{CF} for all
pairs of elliptic curves with conductor less than $500\,000$.

We have no theoretical explanation for our observation that the
surfaces $W(17,1)$ and $W(17,3)$ are birational. It would of course be
interesting to find one. We note that a wealth of information about
the complex geometry of the surfaces $Z(n,k)$ was computed by Kani and
Schanz \cite{KS}. In particular the surfaces $Z(17,1)$ and $Z(17,3)$
are surfaces of general type with geometric genus $10$.  We do not
expect that these surfaces are birational.

In Section~\ref{sec:verify} we verify the $17$-congruences claimed
above by comparing traces of Frobenius mod $17$.  In
Section~\ref{sec:red} we compute a genus $2$ curve whose Jacobian is
isogenous to $E_1 \times E_2$, and note that this gives another proof
that $E_1$ and $E_2$ are $17$-congruent.  We construct our birational
models for $Z(17,1)$ and $Z(17,3)$ as quotients of $X(17) \times
X(17)$. In Section~\ref{sec:X17} we give explicit equations for
$X(17)$, and in Sections~\ref{comp:symp} and~\ref{sec:anti} we compute
the quotients in the symplectic and anti-symplectic cases. In the
final two sections we describe some of the interesting curves and
points that we have so far found on these surfaces.

\section{Verification via modularity} 
\label{sec:verify}

In \cite[p.133]{M} Mazur asked whether there are any non-trivial
symplectic $n$-congruences for any integer $n \ge 7$. The question was
answered by Kraus and Oesterl\'e \cite{KO} who gave the example of the
pair of symplectically $7$-congruent elliptic curves $152a1$ and
$7448e1$. (We use the subsequent labelling of these curves in
Cremona's tables.) They also established the following results.

\begin{Lemma} 
\label{lem1}
\cite[Proposition 2]{KO} Let $p$ be a prime number. Let $E$ and $E'$
by $p$-congruent elliptic curves over $\Q$, with minimal discriminants
$\Delta$ and $\Delta'$. Suppose that $E$ and $E'$ have multiplicative
reduction at a prime $\ell \not= p$ and that the exponent
$v_\ell(\Delta)$ is coprime to $p$.  Then $v_\ell(\Delta')$ is coprime
to $p$, and the $p$-congruence is symplectic if and only if the ratio
$v_\ell(\Delta)/v_\ell(\Delta')$ is a square mod $p$.
\end{Lemma}

\begin{Lemma} 
\label{lem2}
\cite[Proposition 4]{KO} Let $E$ and $E'$ be modular elliptic curves
over $\Q$ with conductors $N$ and $N'$.  Let $S$ be the set of primes
for which one of the curves has split multiplicative reduction, and
the other has non-split multiplicative reduction. Let $M = \lcm(N,N')
\prod_{\ell \in S} \ell$ and \[\mu(M) = [\SL_2(\Z): \Gamma_0(M)] = \#
\PP^1(\Z/M\Z) = M \prod_{\ell | M} (1 + \ell^{-1}).\] Then the
following conditions are equivalent.
\begin{enumerate}
\item The Galois modules $E[p]$ and $E'[p]$ have isomorphic
  semi-simplifications.
\item $a_\ell(E) \equiv a_\ell(E') \pmod{p}$ for all primes $\ell <
  \mu(M)/6$ with $v_\ell (N N') = 0$; and $a_\ell(E) a_\ell(E') \equiv
  \ell + 1 \pmod{p}$ for all primes $\ell < \mu(M)/6$ with $v_\ell (N
  N') = 1$.
\end{enumerate}
\end{Lemma}

Since $X_0(17)$ is a rank $0$ elliptic curve, there are only finitely
many $j$-invariants of elliptic curves over $\Q$ admitting a rational
$17$-isogeny. As noted in \cite[Section~3.8]{Cr}, the exceptional
$j$-invariants are $-17^2 \cdot 101^3/2$ and $-17 \cdot 373^3/2^{17}$.
Ignoring these two $j$-invariants, the conclusion of
Lemma~\ref{lem2}(i), when $p=17$, is that $E$ and $E'$ are
$17$-congruent.

\begin{Example}
\label{ex1}
Let $E_1$ and $E_2$ be the elliptic curves defined in the
introduction.  For the primes $\ell < 50$ the traces of Frobenius are
as follows.
\[ \begin{array}{c|ccccccccccccccc}
\ell & 2 & 3 & 5 & 7 & 11 & 13 & 17 & 19 & 23 & 29 & 31 & 37 & 41 & 43 & 47 \\ \hline
a_\ell(E_1) & -1 & 0 & 1 & -1 & -5 & 1 & 2 & 4 & -1 & -3 & -2 & -11 & 5 & -4 & -9 \\ 
a_\ell(E_2) &  -1 & 0 & 1 & -1 & 1 & 1 & 2 & 4 & -1 & -3 & -2 & 6 & -12 & -4 & -9 
\end{array} \]
In the notation of Lemma~\ref{lem2} we have $S = \emptyset$ and $M =
N(E_2)$.  It takes about three hours\footnote{Running Magma on a
  single core of the author's desktop.} to verify that $a_\ell(E_1)
\equiv a_\ell(E_2) \pmod{17}$ for all primes $\ell < \mu(M)/6 \approx
1.033 \times 2^{30}$ with $\ell \not= 11$.  This shows that $E_1$ and
$E_2$ are $17$-congruent. Since
\begin{align*}
\Delta(E_1) &= 2^3 \cdot 5^3 \cdot 13 \cdot 59^2 \cdot 191^3, \\
\Delta(E_2) &= -2^{14} \cdot 5^{11} \cdot 11^{17} \cdot 13 \cdot 59 \cdot 191^9,
\end{align*}
it follows by Lemma~\ref{lem1} (with $\ell = 2,5,13$ or $59$) that the
congruence is symplectic.
\end{Example}

\begin{Example}
Let $E'_1$ and $E'_2$ be the elliptic curves defined in the
introduction.  For the primes $\ell < 50$ the traces of Frobenius are
as follows.
\[ \begin{array}{c|ccccccccccccccc}
\ell & 2 & 3 & 5 & 7 & 11 & 13 & 17 & 19 & 23 & 29 & 31 & 37 & 41 & 43 & 47 \\ \hline
a_\ell(E'_1) & -1 & 1 & 0 & 0 & 0 & -3 & 2 & -1 & -2 & -8 & 8 & -7 & 0 & 8 & -10 \\
a_\ell(E'_2) & -1 & 1 & 0 & 0 & 0 & 1 & 2 & -1 & -2 & 9 & -9 & 10 & 0 & 8 & 7 
\end{array} \]
In the notation of Lemma~\ref{lem2} we have $S = \emptyset$ and $M =
N(E'_2)$.  It takes a fraction of a second to verify that $a_\ell(E_1)
\equiv a_\ell(E_2) \pmod{17}$ for all primes $\ell < \mu(M)/6 = 15680$
with $\ell \not= 13$.  This shows that $E'_1$ and $E'_2$ are
$17$-congruent. Since
\begin{align*}
\Delta(E'_1) &= -3^5 \cdot 5^2 \cdot 7^2, \\
\Delta(E'_2) &= -3^2 \cdot 5^2 \cdot 7^2 \cdot 13^{17},
\end{align*}
it follows by Lemma~\ref{lem1} (with $\ell = 3$) that the
congruence is anti-symplectic.
\end{Example}

\begin{Remark} (i) It may be possible to reduce the Sturm bound,
  and hence the runtime, in Example~\ref{ex1} by using a result
  similar to \cite[Theorem~9.21]{Stein} or by using level lowering.
  We did not pursue this. \\
  (ii) Methods for determining the symplectic type in situations where
  Lemma~\ref{lem1} does not apply have recently been studied in
  \cite{CF,FK}. \\
  (iii) The existence of a prime $\ell$ for which Lemma~\ref{lem1}
  applies, together with the Weil pairing and the fact our elliptic
  curves do not admit a rational $17$-isogeny, is already enough (see
  \cite[Chapter IV, Section 3.2]{SerreMcGill}) to show that in each
  case the mod $17$ Galois representation is surjective.
\end{Remark}

\section{Verification via genus $2$ Jacobians} 
\label{sec:red}

Let $E_1$ and $E_2$ be $n$-congruent elliptic curves over $\Q$, where
the congruence $\psi$ reverses the sign of the Weil pairing. Then the
quotient $J$ of $E_1 \times E_2$ by the graph of $\psi$ is a
principally polarised abelian surface.  It is shown in \cite[Section
  1]{FreyKani} that if $n$ is odd and $E_1$ and $E_2$ are not
isogenous, then $J$ is the Jacobian of a genus $2$ curve $C$ defined
over $\Q$, and there are degree $n$ morphisms $C \to E_1$ and $C \to
E_2$, also defined over $\Q$. For further details of this construction
of reducible genus $2$ Jacobians, see for
example~\cite{BHLS,BD-(44)-splitting,FreyKani,Kumar,Kuhn,Shaska}.

Since $-1$ is a quadratic residue mod $17$, the elliptic curves $E_1$
and $E_2$ defined in the introduction are of the form considered in
the last paragraph.  In this section we compute the corresponding
genus $2$ curve, and note that this gives another proof that $E_1$ and
$E_2$ are $17$-congruent.


\begin{Lemma}
\label{get-tau}
Let $E_1 = \C/(\Z + \tau_1 \Z)$ and $E_2 = \C/(\Z + \tau_2 \Z)$ with
$\Im(\tau_1), \Im(\tau_2)>0$. Let $\psi: E_1[n] \to E_2[n]$ be the
isomorphism given by $\frac{1}{n}(r + s \tau_1) \mapsto \frac{1}{n}(r
- s \tau_2)$ for $r,s = 0,1,\ldots,n-1$. (Note the minus sign!) Then
the quotient $J$ of $E_1 \times E_2$ by the graph of $\psi$ is
represented in the Siegel upper half-space by
\[   \tau = \begin{pmatrix} n \tau_1 & \tau_1 \\
 \tau_1 & (\tau_1 + \tau_2)/n \end{pmatrix}. \]
\end{Lemma}
\begin{proof}
We have $J \isom \C^2/\Lambda$ where $\Lambda$ is the lattice
spanned by the columns $b_1, \ldots, b_4$ of the matrix 
\[ \begin{pmatrix} \tau_1 & \tau_1/n & 1/n & 0 \\
0 & -\tau_2/n & 1/n & -1 \end{pmatrix}. \] The principal polarisation
on $J$ is given by the Hermitian Riemann form
\[ H((z_1,z_2),(w_1,w_2)) = n \left( 
\frac{z_1 \overline{w}_1}{\Im(\tau_1)} 
+ \frac{z_2 \overline{w}_2}{\Im(\tau_2)} \right), \]  
whose imaginary part is given with respect to the basis 
$b_1, \ldots, b_4$ for $\Lambda$ by
\[ \begin{pmatrix} 0 & I_2 \\ -I_2 & 0 \end{pmatrix}. \]
We therefore take
\[ \tau = ( b_3 | b_4 )^{-1} ( b_1 | b_2) = 
\begin{pmatrix} n & 0 \\ 1 & -1 \end{pmatrix}
\begin{pmatrix} \tau_1 & \tau_1/n \\ 0 & -\tau_2/n \end{pmatrix}
= \begin{pmatrix} n \tau_1 & \tau_1 \\
 \tau_1 & (\tau_1 + \tau_2)/n \end{pmatrix}. \qedhere \]
\end{proof}

The elliptic curves $E_1$ and $E_2$ defined in the introduction are
represented in the upper half-plane by $\tau_1 \approx 0.1142862335 i$
and $\tau_2 \approx 0.5000000000 + 1.415897663 i$.  We computed $\tau$
in the Siegel upper half-space corresponding to our genus $2$ Jacobian
by applying Lemma~\ref{get-tau} (with $n=17$) to $\tau_1$ and $\tau'_2
= (64 \tau_2 - 15)/(-17 \tau_2 + 4)$.  The formula for $\tau'_2$ had
to be guessed, but since the congruence must respect complex
conjugation there were only $8$ possibilities to try.  We then used
the methods described by van Wamelen \cite{Wamelen} to compute the
Igusa-Clebsch invariants of our genus~$2$ curve to 300 decimal digits
of precision. Recognising these as rational numbers, we next used the
method of Mestre \cite{Mestre} to find a genus $2$ curve over $\Q$
with these invariants.  Up to quadratic twist, this gave the genus $2$
curve $C$ with equation $y^2 = f_1(x) f_2(x)$ where
\begin{align*}
f_1(x) &= 196081931 x^3 + 1143338037 x^2 - 801791940 x + 135616700, \\
f_2(x) &= -25996 x^3 + 1698260 x^2 - 6845267 x + 3822078.
\end{align*}
We chose this particular quadratic twist since it satisfies $\#\Jac(C)(\F_p)
= \# E_1(\F_p) \cdot \#E_2(\F_p)$ for many primes $p$ of good
reduction.

To prove that our equation for $C$ is correct (without relying on the
numerical approximations in the last paragraph) we also computed the
degree $17$ morphisms $\phi_1 : C \to E_1$ and $\phi_2 : C \to E_2$.
The $x$-coordinate of $\phi_i$ is given by $\xi_i(x) = h_i(x)/(f_i(x)
g_i(x)^2)$ where $g_i$ and $h_i$ are certain polynomials of degrees
$7$ and $17$. Working mod $p = 101$ we find
\begin{align*}
g_1(x) &= 25 x^7 + 56 x^6 + 31 x^5 + 99 x^4 + 100 x^3 + 42 x^2 + 79 x + 5, \\
g_2(x) &= 3 x^7 + 76 x^6 + 44 x^5 + 97 x^4 + 52 x^3 + 38 x^2 + 75 x + 2, \\
h_1(x) &= 16 x^{17} + 6 x^{16} + 57 x^{15} + 54 x^{14} + 94 x^{13}
  + 79 x^{12} + 77 x^{11} + 55 x^{10} \\ & \quad + 74 x^9 + 78 x^8 + 97 x^7
  + 79 x^6 + 25 x^5 + 96 x^4 + 98 x^3 + 46 x^2 + 4 x + 99, \\ 
  h_2(x) &= 67 x^{17} + 25 x^{16} + x^{15} + 22 x^{14} + 84 x^{13}
  + 94 x^{12} + 93 x^{11} + 95 x^{10} \\ & \quad + 34 x^9 + 40 x^8 + 99 x^7
  + 84 x^6 + 43 x^5 + 12 x^4 + 59 x^3 + 13 x^2 + 26 x + 98.
\end{align*}

The full expressions for $g_1,g_2,h_1,h_2 \in \Z[x]$ may be found in
\cite{magmafile}. We do not record these here, since some of the
coefficients have nearly 100 decimal digits.

Our method to compute these polynomials was to compute them mod $p$
for many primes $p$ and then use the Chinese remainder theorem.  To
compute them mod $p$ we looped over all possibilities for the map
$C(\F_p) \to E_i(\F_p)$, compatible with the group laws on the
Jacobians, and then solved for the rational function $\xi_i$ (with
numerator and denominator of degree at most $17$) by interpolation.

The $y$-coordinates of the maps $\phi_i : C \to E_i$ are of course
even more complicated to write down. However, a convenient alternative
to recording these directly is to note that the invariant
differentials on $E_1$ and $E_2$ pull back to the following ``elliptic
differentials'' on $C$:
\begin{align*}
\phi_1^*\left(\frac{dx}{2y + x}\right) &= \frac{(273857 x - 336364)dx}{y}, \\
\phi_2^*\left(\frac{dx}{2y + x}\right) &= \frac{(2758 x + 1630)dx}{y}.
\end{align*}

Our second proof that $E_1$ and $E_2$ are $17$-congruent is completed
by the next lemma, which we record for convenience, but is essentially
well known.  Compared to the proof in Section~\ref{sec:verify}, this
proof takes a fraction of the computer time, since we only have to
check that our formulae for $\phi_1$ and $\phi_2$ do indeed define
morphisms $C \to E_1$ and $C \to E_2$.

\begin{Lemma}
Let $C$ be a genus $2$ curve and let $p$ be a prime.  Let $\phi_1 : C
\to E_1$ and $\phi_2: C \to E_2$ be morphisms of degree $p$, where
$E_1$ and $E_2$ are non-isogenous elliptic curves.  Then $E_1$ and
$E_2$ are $p$-congruent.
\end{Lemma}
\begin{proof}
Since $E_1$, $E_2$ and $J = \Jac C$ are principally polarised abelian
varieties, we identify them with their duals without further comment.

The map $\phi_1 : C \to E_1$ induces by pull back a map $E_1 \to J$.
This map is injective since otherwise, by
\cite[Proposition~11.4.3]{BL}, $\phi_1$ would have to factor via a
non-trivial isogeny of elliptic curves, which is not possible by our
assumption that $\phi_1$ has prime degree.  Since $E_1$ and $E_2$ are
not isogenous, the composite of the maps $E_1 \to J$ and $J \to E_2$
induced by $\phi_1$ and $\phi_2$ must be the zero map.  The same
observations apply with the roles of $E_1$ and $E_2$ swapped over. The
pull back and push forward maps associated to $\phi_1$ and $\phi_2$
therefore define dual isogenies
\[ E_1 \times E_2 \stackrel{\phihat}{\ra} J \stackrel{\phi}{\ra} 
E_1 \times E_2 \]
whose composite is multiplication-by-$p$. In particular $\deg \phi
 = \deg \phihat = p^2$ and there are isomorphisms of Galois modules
$E_1[p] \isom J[\phi] \isom E_2[p]$.
\end{proof}

\section{The modular curve $X(17)$}
\label{sec:X17} 

Let $\zeta = e^{2\pi i/17}$ and $\xi_k = \zeta^k + \zeta^{-k}$.  Let
$G \isom \PSL_2(\Z/17\Z)$ be the subgroup of $\SL_9(\C)$ generated by
$M_2$ and $M_{17}$ where
\[ M_2 = \frac{-1}{\sqrt{17}} \begin{pmatrix} 
 1&   1&     1&     1&     1&     1&     1&     1&     1 \\
 2&\xi_{3}&\xi_{8}&\xi_{7}&\xi_{4}&\xi_{5}&\xi_{2}&\xi_{6}&\xi_{1} \\
 2&\xi_{8}&\xi_{7}&\xi_{4}&\xi_{5}&\xi_{2}&\xi_{6}&\xi_{1}&\xi_{3} \\
 2&\xi_{7}&\xi_{4}&\xi_{5}&\xi_{2}&\xi_{6}&\xi_{1}&\xi_{3}&\xi_{8} \\
 2&\xi_{4}&\xi_{5}&\xi_{2}&\xi_{6}&\xi_{1}&\xi_{3}&\xi_{8}&\xi_{7} \\
 2&\xi_{5}&\xi_{2}&\xi_{6}&\xi_{1}&\xi_{3}&\xi_{8}&\xi_{7}&\xi_{4} \\
 2&\xi_{2}&\xi_{6}&\xi_{1}&\xi_{3}&\xi_{8}&\xi_{7}&\xi_{4}&\xi_{5} \\
 2&\xi_{6}&\xi_{1}&\xi_{3}&\xi_{8}&\xi_{7}&\xi_{4}&\xi_{5}&\xi_{2} \\
 2&\xi_{1}&\xi_{3}&\xi_{8}&\xi_{7}&\xi_{4}&\xi_{5}&\xi_{2}&\xi_{6}
\end{pmatrix} \]
and $M_{17} = \Diag(1, \zeta, \zeta^9, \zeta^{13}, \zeta^{15}, 
                   \zeta^{16}, \zeta^8, \zeta^4, \zeta^2)$.
The pattern of subscripts in the definition of $M_2$ is the 
sequence of powers of $3$ in $(\Z/17\Z)/\{\pm 1\}$.

We write $\C[x_0,\ldots,x_8]_d$ for the space of homogeneous
polynomials of degree $d$.  An {\em invariant} of degree $d$ is a
polynomial $I \in \C[x_0,\ldots,x_8]_d$ satisfying $I \circ g = I$ for
all $g \in G$.  In degrees $2$ and $3$ the only invariants are
\begin{align*}
Q &= x_0^2 + x_1 x_5 + x_2 x_6 + x_3 x_7 + x_4 x_8, \\
D &= 2 x_0 (x_1 x_5 - x_2 x_6 + x_3 x_7 - x_4 x_8) \\
                  & \qquad \qquad - x_1^2 x_4 + x_2^2 x_5 - x_3^2 x_6 + x_4^2 x_7 
                  - x_5^2 x_8 + x_6^2 x_1  - x_7^2 x_2 + x_8^2 x_3. 
\end{align*}
In degree $4$ we have the invariants $Q^2$ and 
\begin{align*}
F &= x_0^4 + x_0 (x_1^2 x_4 + x_2^2 x_5 + x_3^2 x_6 + x_4^2 x_7
  + x_5^2 x_8 + x_1 x_6^2 + x_2 x_7^2 + x_3 x_8^2) \\
  &+ x_1 x_3 x_5 x_7 + x_2 x_4 x_6 x_8
  + x_1 x_2 x_5 x_6 + x_2 x_3 x_6 x_7 + x_3 x_4 x_7 x_8 + x_1 x_4 x_5 x_8
  \\ & + x_1^2 x_3 x_8 + x_1 x_2^2 x_4 + x_2 x_3^2 x_5 + x_3 x_4^2 x_6
  + x_4 x_5^2 x_7 + x_5 x_6^2 x_8  + x_1 x_6 x_7^2 + x_2 x_7 x_8^2.
\end{align*}

\begin{Proposition}
\label{prop:96+168}
Let $C \subset \PP^8$ be the curve defined by the vanishing of $Q$ and
all partial derivatives of $F$. Then $C = C_1 \cup C_2$ where $C_1$
and $C_2$ are curves of degrees $96$ and $168$, each isomorphic to the
modular curve $X(17)$. The $144$ cusps on $C_1$ are cut out (each with
multiplicity $2$) by the cubic form $D$. Moreover $D$ vanishes
identically on $C_2$.
\end{Proposition}

\begin{proof}
Let $p \ge 5$ be a prime. The group $\PSL_2(\Z/p\Z)$ acts on $X(p)$
with quotient the $j$-line, and the group of divisor classes fixed by
this group action is an infinite cyclic group, generated by $\lambda$
of degree $(p^2 - 1)/24$. Let $m = (p-1)/2$. Klein gave equations for
$X(p)$ embedded in $\PP^{m-1}$ and $\PP^{m}$ with hyperplane sections
$(m-1)\lambda$ and $m\lambda$. Following \cite{AR} we call these
models the $z$-curve and the $A$-curve.  See \cite[Section 24]{AR} or
\cite[Section 4]{congr13} for further details.

We take $p=17$. Let $z_1, \ldots, z_8$ be coordinates on $\PP^7$. We
write $z_0 = 0$, $z_{-i} = -z_{i}$ and agree to read all subscripts
mod $17$. According to \cite[Section~2]{7and11} the $z$-curve for
$X(17)$ is the curve in $\PP^7$ defined by the 4 by 4 Pfaffians of the
17 by 17 skew symmetric matrix $(z_{i-j} z_{i+j})$. We define maps
$\phi_i : X(17) \to \PP^8$ for $i=1,2$ by
\[ \phi_1 = 
\left( 1 : \frac{z_2}{z_1} : \frac{z_6}{z_3} : \frac{-z_1}{z_8} : \frac{-z_3}{z_7} : 
\frac{z_8}{z_4} : \frac{-z_7}{z_5} : \frac{z_4}{z_2} : \frac{-z_5}{z_6} \right) \]
and
\begin{align*} \phi_2 &= \bigg(
    z_1 z_4 + z_2 z_8 + z_3 z_5 - z_6 z_7 - \frac{ 2 z_4 z_7 z_2}{z_1} \\ & :
    -z_8^2 - \frac{2 z_5 z_7 z_2}{z_1} :
    z_7^2 - \frac{2 z_2 z_4 z_6}{z_3} : 
    -z_4^2 + \frac{2 z_5 z_6 z_1}{z_8} : 
    z_5^2 + \frac{2 z_1 z_2 z_3}{z_7} \\ & :
    -z_2^2 + \frac{2 z_3 z_6 z_8}{z_4} :
    z_6^2 - \frac{2 z_1 z_8 z_7}{z_5} :
    -z_1^2 - \frac{2 z_3 z_7 z_4}{z_2} :
    z_3^2 + \frac{2 z_4 z_8 z_5}{z_6} \bigg).
\end{align*}

Let $C_1$ and $C_2$ be the images of $\phi_1$ and $\phi_2$.  We find
using Magma \cite{Magma} that $Q$, the partial derivatives of $F$, and
the quartics $x_0^4 + x_1 x_3 x_5 x_7$ and $x_0^4 + x_2 x_4 x_6 x_8$
vanish on $C_1$. Likewise, $Q$, $D$, and the partial derivatives of
$F$ vanish on $C_2$. These equations are sufficient to define each
curve set-theoretically. In fact, the homogeneous ideal of $C_1$ is
generated by one quadratic form, $9$ cubic forms and $117$ quartic
forms, and the homogeneous ideal of $C_2$ is generated by one
quadratic form and $28$ cubic forms.  To prove the decomposition $C =
C_1 \cup C_2$ we checked that $(x_0^4 + x_1 x_3 x_5 x_7)D^2$ and
$(x_0^4 + x_2 x_4 x_6 x_8)D^2$ belong to the ideal generated by $Q$
and the partial derivatives of $F$.

Since $Q$ and $F$ are invariants, the group $G$ acts on $C$, and hence
on $C_1$ and $C_2$. It is shown in \cite[Lemma 20.40]{AR} that for $p
\ge 7$ a prime, the curve $X(p)$ has automorphism group
$\PSL_2(\Z/p\Z)$. So up to an automorphism of $G$, the $G$-actions on
$C_1$ and $C_2$ correspond to the usual action of $\PSL_2(\Z/17\Z)$ on
$X(17)$.  The points on $X(17)$ above $j=0,1728,\infty$ form
$G$-orbits of sizes $816$, $1224$, $144$. All other $G$-orbits have
size $|G|=2448$.  The intersection of $C_1$ with $\{D = 0\}$ has $3
\times 96 = 288$ points counted with multiplicity. Being preserved by
the $G$-action, it must therefore be the set of cusps, each counted
with multiplicity $2$.
\end{proof}

\begin{Remark}
The formula for $\phi_1$ (up to signs and ordering) is that given in
\cite[Section 51]{AR}, and accordingly $C_1$ is the $A$-curve.  The
formula for $\phi_2$ was found by using the $G$-actions to compute a
complement to the image of $S^2{\mathcal L}(\zeta)$ in ${\mathcal
  L}(2\zeta)$, where ${\mathcal L(\delta)}$ denotes the Riemann Roch
space of a divisor $\delta$, and $\zeta \sim 7\lambda$ is the
hyperplane section for the $z$-curve.
\end{Remark}

\begin{Definition}
\label{def:cov}
A {\em covariant} of degree $d$ is a column vector $\vv$ of
polynomials in $\C[x_0,\ldots,x_8]_d$ satisfying $\vv \circ g = g \vv$
for all $g \in G$.
\end{Definition}

Starting from an invariant $I$ of degree $d$ we may construct a
covariant of degree $d-1$ as
\begin{equation}
\label{nabla}
 \nabla_Q I = H(Q)^{-1} \begin{pmatrix} \partial I/ \partial x_0 \\
   \vdots \\ \partial I/ \partial x_8 \end{pmatrix}
\end{equation}
where $H(Q)$ is the 9 by 9 matrix of second partial derivatives of
$Q$.  Going in the other direction, if $\vv$ and $\ww$ are covariants
of degrees $d$ and $e$ then
\begin{equation}
\label{dot}
\vv \cdot \ww := \vv^T H(Q) \ww = \coeff( Q(\vv + t \ww) , t)
\end{equation}
is an invariant of degree $d+e$. If we think of a covariant as a
$G$-equivariant polynomial map $\C^9 \to \C^9$ then the composition of
covariants $\vv$ and $\ww$ of degrees $d$ and $e$ is a covariant $\vv
\circ \ww$ of degree $de$.

We put $\vv_1 = (x_0,\ldots,x_8)^T$, $\vv_2 = \nabla_Q D$, $\vv_3 =
\nabla_Q F$, $\vv_4 = \vv_2 \circ \vv_2$ and $\vv_6 = \vv_3 \circ
\vv_2$. Then $c_4 = \vv_4 \cdot \vv_6$ is an invariant of degree $10$.

\begin{Lemma} 
\label{prop:j}
Let $X = X(17)$ be the curve denoted $C_1$ in
Proposition~\ref{prop:96+168}. Then the $j$-map $X \to \PP^1$ is given
by $j = -2^7 c_4^3/ D^{10}$.
\end{Lemma}
\begin{proof} The calculation at the end of this proof shows that
  $c_4$ does not vanish identically on $X$. So the intersection of $X$
  with $\{c_4 = 0\}$ is a set of $10 \times 96 = 816 + 144$ points.
  Arguing as in the proof of Proposition~\ref{prop:96+168}, this set
  is the union of the points above $j = 0$ and $j = \infty$. Our
  formula for the $j$-invariant therefore has the correct divisor, and
  so is correct up to scaling.

  Let $u = (\theta^3 - \theta^2 - \theta + 2 i - 1)/4$ where $\theta =
  \sqrt[4]{1-4i}$ and $i = \sqrt{-1}$. Let $\sigma$ be the generator
  for $\Gal(\Q(\theta)/\Q(i))$ given by $\sigma(\theta) = i
  \theta$. The point
  \[ (1 : u:\sigma(u):\sigma^2(u):\sigma^3(u):
  u:\sigma(u):\sigma^2(u):\sigma^3(u) ) \in X \]
  is fixed by a permutation matrix of order 2 in $G$, and so lies
  above $j = 1728$. The function $c_4^3/ D^{10}$ takes the value
  $-1728/2^7$ at this point.
\end{proof}

\section{Computations in the symplectic case}
\label{comp:symp}

Let $X = X(17) \subset \PP^8$ be the curve denoted $C_1$ in
Proposition~\ref{prop:96+168}. By \cite[Lemma 3.2]{congr13} the
surface $Z(17,1)$ is birational to the quotient of $X \times X \subset
\PP^8 \times \PP^8$ by the diagonal action of $G \isom
\PSL_2(\Z/17\Z)$.  We write $x_0, \ldots, x_8$ and $y_0, \ldots, y_8$
for our coordinates on the first and second copies of $\PP^8$.

\begin{Definition}
  A {\em bi-invariant} of degree $(m,n)$ is a polynomial in $x_0,
  \ldots, x_8$ and $y_0,\ldots,y_8$, that is homogeneous of degrees
  $m$ and $n$ in the two sets of variables, and is invariant under the
  diagonal action of $G$.
\end{Definition}

In principle we may find equations for $Z(17,1)$ by computing
generators and relations for the ring of bi-invariants mod $I(X \times
X)$.  In practice we find it is sufficient to compute only some of the
generators and some of the relations.

The calculations that follow rely on showing that certain
bi-invariants vanish identically on $X \times X$. Initially we only
checked that they vanish at many $\F_p$-points for some moderately
large prime $p$.  For a full proof in characteristic $0$ we used the
$G$-action and numerical approximations to verify the conditions
in the following lemma. This is sufficient since if the absolute value of
the norm of an algebraic integer is less than one, then it must be zero.
\begin{Lemma}
\label{lem:23}
Let $I$ be a bihomogeneous form of degree $(m,n)$ with $m,n \le 22$.
If $I$ vanishes at all points $(P,Q) \in X \times X$ with $j(P),j(Q)
\in \{0,1728,\infty\}$ then $I$ vanishes on $X \times X$.
\end{Lemma}
\begin{proof}
This follows from Bezout's theorem, using that $144 + 816 + 1224 > 22
\times 96$. The argument is then identical to that used in the proof of
\cite[Lemma 6.2]{congr13}.
\end{proof}

It is easy to compute the dimension of the space of bi-invariants of
any given degree from the character table of $G$. However we need to
work with explicit bases for these spaces.  Let $Q$, $D$ and $F$ be
the invariants of degrees 2, 3 and 4 defined in
Section~\ref{sec:X17}. We define bi-invariants $Q_{ij}$, $D_{ij}$ and
$F_{ij}$ by the rules
\begin{align*}
  Q(\lambda x_0 + \mu y_0, \ldots , \lambda x_8 + \mu y_8)
    & = \lambda^2 Q_{20} + \lambda \mu Q_{11} + \mu^2 Q_{02}, \\ 
  D(\lambda x_0 + \mu y_0, \ldots , \lambda x_8 + \mu y_8)
    & = \lambda^3 D_{30} + \lambda^2 \mu D_{21} + \ldots
  + \mu^3 D_{03}, \\ 
  F(\lambda x_0 + \mu y_0, \ldots , \lambda x_8 + \mu y_8)
    & = \lambda^4 F_{40} + \lambda^3 \mu F_{31} + \ldots + \mu^4 F_{04}.
\end{align*}
Then, writing $H$ for the $9$ by $9$ matrix of second partial
derivatives with respect to $x_0, \ldots,x_8$, we put
\[ D_x = H(Q)^{-1} H(D_{30}), \,\,\, D_y = H(Q)^{-1} H(D_{21}),
   \,\,\, F_{xy} = H(Q)^{-1} H(F_{31}). \]

The space of bi-invariants of degree $(2,2)$ has dimension $4$, with
basis $A_1, \ldots, A_4$ where $A_1 = Q_{20} Q_{02}$, $A_2 =
Q_{11}^2$, $A_3 = F_{22}$, and
\[ \tr(D_x D_y F_{xy}) = -16 A_1 + 8 A_2 - 8 A_4. \]
Each of these bi-invariants is symmetric (under interchanging
the $x$'s and $y$'s), but only the first vanishes identically
on $X \times X$.

The space of bi-invariants of degree $(3,3)$ has dimensional $16$.
Under interchanging the $x$'s and $y$'s, this breaks up as the direct
sum of symmetric and skew-symmetric subspaces of dimensions $14$ and
$2$. The subspace of symmetric bi-invariants has basis $B'_1, \ldots,
B'_{14}$ given by
\begin{align*}
 & Q_{11} A_2, \,\, Q_{11} A_3, \,\, Q_{11} A_4, \,\,
 D_{30} D_{03}, \,\, D_{21} D_{12}, \,\,
 \tr(D_x^3 D_y^3), \,\,
 \tr(D_x^2 D_y D_x D_y^2), \\ & 
 \tr(D_x^2 D_y^2 F_{xy}), \,\,
 \tr(D_x D_y D_x D_y F_{xy}), \,\,
 \tr(D_x D_y F_{xy}^2), \,\,
 \tr(F_{xy}^3), \\
 & Q_{11} A_1, \,\,
 Q_{20} F_{13}  + Q_{02} F_{31}, \,\,
 2 \nabla_Q F(x_0, \ldots, x_8) \cdot \nabla_Q F(y_0, \ldots, y_8),
\end{align*}
where in the final expression (for $B'_{14}$) we use the
notation~\eqref{nabla} and~\eqref{dot}.

We changed our choice of basis for this space of bi-invariants first
so that the bi-invariants themselves have small integer coefficients,
then so that the relations considered below have small integer
coefficients, and finally to facilitate writing down an elliptic
fibration.  To simplify the calculations that follow, we therefore
(with the benefit of hindsight) switch to the basis $B_1, \ldots,
B_{14}$ that is related to the basis $B'_1, \ldots, B'_{14}$ (as
specified in the last paragraph) by the first change of basis matrix
recorded in Appendix~\ref{app:formulae}. In fact we keep the first and
last three basis elements the same, i.e. $B_i = B'_i$ for
$i=1,2,3,12,13,14$.

The subspace of bi-invariants vanishing on $X \times X$ is spanned by
$B_{12},B_{13},B_{14}$.  We write $\I_1 \subset \Q[z_1,\ldots,z_{11}]$
for the ideal generated by all quadratic and cubic forms vanishing on
the image of the map
\[ X \times X \to \PP^{10} \, ; \quad (x_0, \ldots, x_8;y_0, \ldots, y_8) 
\mapsto (B_1 : \ldots : B_{11}). \]
We find that $\I_1$ is minimally generated by $13$ quadratic forms and
$21$ cubic forms. Moreover the subvariety $\Sigma_1 \subset \PP^{10}$
defined by $\I_1$ is a surface of degree $29$.

\begin{Proposition}
\label{prop:Sigma-K3}
The surface $\Sigma_1$ is birational over $\Q$ to the elliptic surface
$\Sigma$ defined by the Weierstrass equation~\eqref{Weqn0} in the
statement of Theorem~\ref{thm1}.
\end{Proposition}

\begin{proof}
The rational map $\Sigma_1 \to \PP^3 \times \Aff^1$ given by
\[(z_1 : \ldots : z_{11}) \mapsto 
((z_1 : z_2 : z_3 : z_4),T) = ((z_1 : z_2 : z_3 : z_4), z_5/z_6)\]
has image satisfying
\begin{align*}
  & 2 T z_1^2 + 3 T z_1 z_2 - (5 T - 2) z_1 z_3
        - 3 T^2 z_1 z_4 + T z_2^2 - 4 T z_2 z_3  - 2 T^2 z_2 z_4\\ &
  \qquad  + 2 (2 T - 1) z_3^2 + T (4 T - 1) z_3 z_4 + T^3 z_4^2 =0,  \\
  & (T + 1)^2 z_1^2 + T (2 T + 1) z_1 z_2 - (T + 1)^2 z_1 z_3
       - T (T + 1)^2 z_1 z_4  
    + T^2 z_2^2 \\ & \qquad - 2 T^2 z_2 z_3 - T^3 z_2 z_4 =0.  
\end{align*}
These same equations define a genus one curve in $\PP^3$
over the function field $\Q(T)$. Making the linear change of coordinates
\begin{align*}
u_1 &= T (T + 1) z_1,  & u_3 &= (T + 1) z_3, \\
u_2 &= T (z_1 + z_2 - 2 z_3 - T z_4), &
u_4 &= T (-z_1 + 2 z_3 + T z_4), 
\end{align*}
gives the simplified quadric intersection
\begin{align*}
   u_1 u_2 + u_1 u_3 + (T + 1) u_2^2 - u_3 u_4 &=0, \\ 
   u_1 u_4 + T u_2^2 - T^2 u_2 u_4 - T u_3 u_4 &=0,
\end{align*}
which in turn is isomorphic to the elliptic curve~\eqref{Weqn0} via
\[ x = \frac{-T u_1}{u_4}, \quad y = \frac{T u_1 (u_1 - u_2 - u_4)}{u_2 u_4}. \]

Composing these maps gives a birational map $\Sigma_1 \to \Sigma$.
The inverse map, represented as an explicit $11$-tuple of elements in
the function field $\Q(\Sigma)$, is recorded in the accompanying
computer file \cite{magmafile}.
\end{proof}

Since the rational map $X \times X \to \Sigma$ is defined by symmetric
bi-invariants, it factors via a rational map $\pi : W(17,1) \to
\Sigma$.  We will see below that $\pi$ is birational, thereby proving
the first part of Theorem~\ref{thm1}.

One way to compute equations for the double cover $Z(17,1) \to
W(17,1)$ is to find a skew-symmetric bi-invariant of degree $(3,3)$
that does not vanish identically on $X \times X$, and then write its
square, modulo $I(X \times X)$, as a quadratic form in $B_1, \ldots,
B_{11}$.  We omit the details, since this calculation is superceded by
the calculation of the $j$-maps, which we do next.

We consider the symmetric bi-invariants $\alpha_1 = D_{30} D_{03}$,
$\alpha_2 = D_{12} D_{21}$, $\alpha_3 = D_{21}^3 D_{03} + D_{12}^3
D_{30}$, $\alpha_4 = c_4(x_0, \ldots, x_8) c_4(y_0, \ldots, y_8)$ and
\[ \alpha_5 = c_4(x_0, \ldots, x_8) D_{03}^3 D_{12} +
c_4(y_0, \ldots, y_8) D_{30}^3 D_{21}, \] of degrees $(m,m)$ for $m =
3,3,6,10,11$. Let $S = \Q[u, v, w, z_4, \ldots, z_{11}]$ be the
graded polynomial ring where the variables have weights
$1,2,2,3, \ldots, 3$.  In the accompanying Magma file~\cite{magmafile}
we record $g_1, \ldots, g_5 \in S$ of weighted degrees $3,3,9,15,15$
and $h_1, \ldots, h_5 \in S$ of weighted degrees $0,0,3,5,4$, such
that each of the bi-invariants
\[ g_i(Q_{11},A_3,A_4,B_4, \ldots,B_{11}) 
- h_i(Q_{11},A_3,A_4,B_4, \ldots,B_{11}) \alpha_i \] vanishes on $X
\times X$. We use these expressions to solve for the $\alpha_i$ as
elements of the function field $\Q(\Sigma)$.  Then the polynomials
$f_1(Y) = Y^2 - \alpha_3 Y + \alpha_1 \alpha_2^3$ and $f_2(Y) = Y^2 -
\alpha_5 Y + \alpha_1^3 \alpha_2 \alpha_4$ have roots defined over the
same quadratic extension of $\Q(\Sigma)$.  Let these roots be
$r_1,s_1$ and $r_2,s_2$.  If we order these roots appropriately then
by Lemma~\ref{prop:j}, and the definition of the $\alpha_i$, we have
\[ j_1 + j_2 =  \frac{-2^{7}}{\alpha_1^9}\left(\frac{r_2^3}{r_1} + 
\frac{s_2^3}{s_1}\right) \qquad \text{ and } \qquad j_1 j_2 =
\frac{2^{14} \alpha_4^3}{\alpha_1^{10}}. \]

Let $\widetilde{\Sigma} \to \Sigma$ be the double cover defined by the
requirement that $\disc f_1$, $\disc f_2$, or $(j_1 + j_2)^2 - 4 j_1
j_2$ is a square. Then the product of $j$-maps $X \to \PP^1$ factors
as
\begin{equation}
\label{ratmaps}
X \times X \ra Z(17,1) \ra  \widetilde{\Sigma} \ra
\PP^1 \times \PP^1. 
\end{equation}
The composite corresponds to a Galois extension of function fields,
with Galois group $G \times G$. Since $G \isom \PSL_2(\Z/17\Z)$ is a
simple group, the diagonal subgroup $\Delta_G \subset G \times G$ is a
maximal subgroup.  Therefore one of the last two maps
in~\eqref{ratmaps} is birational.  However if the last map were
birational, then this would mean that in attempting to quotient out by
$\Delta_G$, we had in fact quotiented out by $G \times G$. To exclude
this possibility we may check, for example, that the rational function
$F_{22}/Q_{11}^2$ on $X \times X$ is not $G \times G$-invariant.

In conclusion, $Z(17,1)$ is birational to $\widetilde{\Sigma}$, and
$W(17,1)$ is birational to $\Sigma$. This completes the proof of
Theorem~\ref{thm1} in the symplectic case.

\begin{Remark} We initially hoped that we might compute 
  $W(17,1)$ using the bi-invariants of degree $(2,2)$, without needing
  those of degree $(3,3)$. In hindsight we see that this is not
  possible, since the map $W(17,1) \to \PP^2$ given by $(A_2:A_3:A_4)$
  is generically $5$-to-$1$. This may be seen by eliminating $z_4$
  from the first quadric intersection in the proof of
  Proposition~\ref{prop:Sigma-K3}, and noting that the resulting
  quartic in $z_1,z_2,z_3$ has degree $5$ in $T$.
\end{Remark}

\section{Computations in the anti-symplectic case}
\label{sec:anti}

In Section~\ref{sec:X17} we defined $G$ as the subgroup of $\SL_9(\C)$
generated by certain matrices with entries in $\Q(\zeta)$, where
$\zeta$ is a primitive $17$th root of unity. Replacing each matrix
entry by its image under the  automorphism
$\zeta \mapsto \zeta^3$ of $\Q(\zeta)$ defines an outer automorphism
$g \mapsto \widetilde{g}$ of $G$.

\begin{Definition}
  A {\em skew bi-invariant} of degree $(m,n)$ is a polynomial in $x_0,
  \ldots, x_8$ and $y_0,\ldots,y_8$, that is homogeneous of degrees
  $m$ and $n$ in the two sets of variables, and is invariant under the
  action of $G$ via $g : (x,y) \mapsto (gx,\widetilde{g}y)$.
\end{Definition}

Since the map $g \mapsto \widetilde{\widetilde{g}}$ is an inner
automorphism of $G$ we see that 
if $f$ is a skew bi-invariant, then so too is
\[ f^\dagger(x_0, \ldots, x_8;y_0, \ldots, y_8) := 
  f(y_0, \ldots, y_8;-x_0, -x_2, -x_3, \ldots , -x_8, -x_1). \]
We have $f^{\dagger \dagger} = f$.

The space of skew bi-invariants of degree $(2,2)$ has basis $A_1, A_2,
A_3$ where $A_1 = Q_{20} Q_{02}$ and
\begin{align*}
A_2 &= 4 x_0^2 y_0^2
   + \textstyle{\sum}
     (2 x_0 x_1 y_3 y_4 + 2 x_4 x_5 y_0 y_1 + x_1^2 y_1 y_7 + x_2 x_8 y_1^2
     + x_1 x_2 y_2 y_5 \\ 
   & \hspace{5em} + x_3 x_6 y_1 y_2 + x_1 x_3 y_5 y_8 + x_1 x_6 y_1 y_3
     + \tfrac{1}{2} (x_1 x_5 y_2 y_6 + x_3 x_7 y_1 y_5)), \\
A_3 &= 4 x_0^2 y_0^2
   + \textstyle{\sum} 
     (2 x_0 x_1 y_2 y_8 + 2 x_1 x_3 y_0 y_1 + x_1^2 y_2 y_3 + x_3 x_4 y_1^2
     + x_1 x_2 y_2 y_5 \\
   & \hspace{5em} + x_3 x_6 y_1 y_2 + x_1 x_3 y_5 y_8 + x_1 x_6 y_1 y_3
     + \tfrac{1}{2} (x_1 x_5 y_1 y_5 + x_2 x_6 y_1 y_5)).
\end{align*}
Here $\sum$ denotes the sum over all simultaneous cyclic permutations
of $x_1, \ldots, x_8$ and $y_1, \ldots, y_8$ (fixing $x_0$ and $y_0$).
We have $A_i^\dagger = A_i$ for $i=1,2,3$.

The space of skew bi-invariants of degree $(3,1)$ is $1$-dimensional,
spanned by
\begin{align*}
S_{31} &= -16 x_0^3 y_0 + \textstyle{\sum}((3 x_0 x_1 x_5 
    + 3 x_1^2 x_4) y_0
    + (6 x_0 x_1 x_3 + 6 x_0 x_4 x_5 + 3 x_2^2 x_3  \\ 
   & \hspace{3em} + 3 x_1 x_4^2 + x_5^3
    + 3 x_4 x_6^2 + 6 x_1 x_2 x_7 + 6 x_3 x_5 x_8 + 6 x_6 x_7 x_8) y_1).
\end{align*}
We write $S_{13} = S_{31}^\dagger$ for the corresponding skew
bi-invariant of degree $(1,3)$.

Earlier we wrote $H(f)$ for the $9$ by $9$ matrix of second partial
derivatives of $f$ with respect to $x_0, \ldots, x_{8}$.  We now write
$H = H_{xx}$ and define $H_{xy}, H_{yx}, H_{yy}$ in the analogous way.
We further put $\H_{xx}(f) = H(Q)^{-1} H_{xx}(f)$, $\H_{xy}(f) =
H(Q)^{-1} H_{xy}(f)$ and so on. The following are skew bi-invariants
of degree $(3,3)$.
\begin{align*}
P &= 2 (Q_{20} S_{13} + Q_{02} S_{31}), \\
\Theta_{ij} &= \tr(\H_{xx}(D_{30}) \H_{xy}(A_i) \H_{yy}(D_{03}) \H_{yx}(A_j)), \\
\Psi_{ij} &= \tr(\H_{xx}(S_{31}) \H_{xy}(A_i) \H_{yx}(A_j)), \\
U_i &= \tr(\H_{xy}(P) \H_{yx}(A_i)), \\
V &= \tr(\H_{xy}(\Theta_{12}) \H_{yx}(A_3)).
\end{align*}
The space of skew bi-invariants of degree $(3,3)$ has dimension $15$.
The subspace fixed by the involution $f \mapsto f^\dagger$ has basis 
$B'_1, \ldots, B'_{12}$ given by
\[ D_{30} D_{03}, \,\, P, \,\,
  \Theta_{12}, \,\, \Theta_{13}, \,\, \Theta_{22}, \,\, 
  \Theta_{23}, \,\, \Theta_{33}, \,\,
  U_2, \,\, U_3, \,\, \Psi_{22} + \Psi_{22}^\dagger,
  \Psi_{23} + \Psi_{23}^\dagger, \,\, V.
  \]
  For the calculations that follow we switch (with the benefit of
  hindsight) to the basis $B_1, \ldots, B_{12}$ that is related to the
  basis $B'_1, \ldots, B'_{12}$ by the second change of basis matrix
  recorded in Appendix~\ref{app:formulae}.

The subspace of bi-invariants vanishing on $X \times X$ has basis
$B_{11}, B_{12}$.  We write $\I_3 \subset \Q[z_1,\ldots,z_{10}]$ for
the ideal generated by all quadratic and cubic forms vanishing on the
image of the map
\[ X \times X \to \PP^{9} \, ; \quad (x_0, \ldots, x_8;y_0, \ldots, y_8) 
\mapsto (B_1 : \ldots : B_{10}). \]
We find that $\I_3$ is minimally generated by $14$ quadratic forms and $2$
cubic forms. Moreover the variety $\Sigma_3 \subset \PP^{9}$ defined by 
$\I_3$ is a surface of degree $24$.

\begin{Proposition}
\label{prop:Sigma-K3-anti}
The surface $\Sigma_3$ is birational over $\Q$ to the elliptic surface
$\Sigma$ defined by the Weierstrass equation~\eqref{Weqn0} in the
statement of Theorem~\ref{thm1}.
\end{Proposition}

\begin{proof}
The rational map $\Sigma_3 \to \PP^3 \times \Aff^1$ given by
\[(z_1 : \ldots : z_{10}) \mapsto 
((z_1 : z_2 : z_3 : z_4),T) = ((z_1 : z_2 : z_3 : z_4), z_5/z_6)\]
has image satisfying
\begin{align*}
 z_1^2 - T z_1 z_2 + T z_1 z_3 + T z_2 z_3 - T z_3^2 +
    T z_3 z_4 = 0, \\
    z_1 z_3 - T z_2 z_3 + 2 (T + 1) z_1 z_4 
    - (T^2 - 1) z_2 z_4 + (T + 1) z_4^2 = 0.
\end{align*}
These same equations define a genus one curve in $\PP^3$
over $\Q(T)$. Making the linear change of coordinates
\begin{align*}
u_1 &= z_1 - T z_2, & u_3 &= (T + 1) z_3, \\
u_2 &= (T + 2) z_1 + z_2 - (T + 1)(z_3 - z_4), 
& u_4 &= (T + 1) z_4 
\end{align*}
gives the simplified quadric intersection
\begin{align*}
u_1^2 + T u_1 u_2 + T u_2 u_3 - T u_1 u_4 &= 0, \\ 
u_1 u_3 + T u_1 u_4 + u_2 u_4 + u_3 u_4 &= 0,
\end{align*}
which in turn is isomorphic to the elliptic curve~\eqref{Weqn0} via
\[ x = \frac{T (u_1 + T u_2)}{u_4}, \quad
y = \frac{T (u_1 + T u_2)^2}{u_1 u_4}. \]

Composing these maps gives a birational map $\Sigma_3 \to \Sigma$.
The inverse map, represented as an explicit $10$-tuple of elements in the
function field $\Q(\Sigma)$, is recorded in the accompanying 
computer file \cite{magmafile}.
\end{proof}

Since the $B_i$ are skew bi-invariants satisfying $B_i^\dagger = B_i$,
the rational map $X \times X \to \Sigma$ factors via a rational map
$\pi : W(17,3) \to \Sigma$.  We will see below that $\pi$ is
birational, thereby proving the first part of Theorem~\ref{thm1}.

For the purpose of computing the $j$-maps we decided to work with the
skew bi-invariants of degree $(2,2)$, alongside those of degree
$(3,3)$. We consider the map from $X \times X$ to the weighted
projective space $\PP(2,2,3, \ldots,3)$ given by
\[ (v,w,z_1, \ldots, z_{10}) = (A_2,A_3,B_1, \ldots, B_{10}). \]
Among the equations defining the image of this map we found the relations
\begin{align*}
 v z_5 z_{10} - w (z_1 z_5 - z_3 z_5 + z_1 z_6 - z_5 z_8 + z_5 z_9) &= 0, \\
 v w^2 z_7 - (z_2 - z_3 - z_5) (z_6 - z_7) z_{10} + z_5 z_8 z_{10} &= 0.
\end{align*}
We used these relations to extend our map $\Sigma \to \PP^9$ to a map
$\Sigma \to \PP(2,2,3, \ldots,3)$.

The space of skew bi-invariants of degree $(3,2)$ has dimension $5$.
We picked one of these skew bi-invariants, not vanishing on $X \times
X$, and called it $T_{32}$. We also put $T_{23} = T_{32}^\dagger$.  We
define skew bi-invariants $\alpha_1 = D_{30} D_{03}$, $\alpha_2 =
S_{31} S_{13}$, $\alpha_3 = T_{32} T_{23}$, $\alpha_4 = D_{03} S_{31}
T_{32} + D_{30} S_{13} T_{23}$, $\alpha_5 = S_{31} T_{23}^2 + S_{13}
T_{32}^2$,
\begin{align*}
  \alpha_6 &= c_4(x_0, \ldots, x_8) c_4(y_0, \ldots, y_8), \\
  \alpha_7 &= D_{03}^3 T_{23} c_4(x_0, \ldots, x_8)
      + D_{30}^3 T_{32} c_4(y_0, \ldots, y_8),
\end{align*}
of degrees $(m,m)$ for $m =3,4,5,6,7,10,12$.  Let $S = \Q[v, w, z_1,
  \ldots, z_{10}]$ be the coordinate ring of $\PP(2,2,3,\ldots,3)$,
i.e., the graded polynomial ring where the variables have these
weights.  In the accompanying Magma file~\cite{magmafile}, we record
$g_1, \ldots, g_7 \in S$ of weighted degrees $3, 6, 8, 8, 12, 18, 18$,
and $h_1, \ldots, h_7 \in S$ of weighted degrees $0, 2, 3, 2, 5, 8,
6$, such that each of the skew bi-invariants
\[ g_i(A_2,A_3,B_1, \ldots,B_{10}) 
-  h_i(A_2,A_3,B_1, \ldots,B_{10}) \alpha_i \]
vanishes on $X \times X$. We use these expressions to solve for
the $\alpha_i$ as elements of the function field $\Q(\Sigma)$. 
Then the polynomials 
\begin{align*}
f_1(Y) &= Y^2 - \alpha_4 Y + \alpha_1 \alpha_2 \alpha_3, \\
f_2(Y) &= Y^2 - \alpha_5 Y + \alpha_2 \alpha_3^2, \\
f_3(Y) &= Y^2 - \alpha_7 Y + \alpha_1^3 \alpha_3 \alpha_6,
\end{align*}
have roots defined over the same quadratic extension of $\Q(\Sigma)$.
Let $f_i$ have roots $r_i,s_i$.  If we order these roots appropriately
then by Lemma~\ref{prop:j}, and the definition of the $\alpha_i$, we
have
\[ j_1 + j_2 = \frac{-2^{7} \alpha_{2}}{\alpha_{1}^9}
 \left( \frac{r_3^3}{r_1 r_2} +\frac{s_3^3}{s_1 s_2} \right) 
\quad \text{ and } \quad
j_1 j_2 = \frac{2^{14} \alpha_{6}^3}{\alpha_1^{10}}. \]

Let $\widetilde{\Sigma} \to \Sigma$ be the double cover defined by the
requirement that $\disc f_1$, $\disc f_2$, $\disc f_3$ or $(j_1 +
j_2)^2 - 4 j_1 j_2$ is a square. Then the product of $j$-maps $X \to
\PP^1$ factors as
\begin{equation*}
X \times X \ra Z(17,3) \ra  \widetilde{\Sigma} \ra
\PP^1 \times \PP^1. 
\end{equation*}
Exactly as in Section~\ref{comp:symp} it follows that $Z(17,3)$ is
birational to $\widetilde{\Sigma}$ and $W(17,3)$ is birational to
$\Sigma$.  This completes the proof of Theorem~\ref{thm1} in the
anti-symplectic case.

\begin{Remark} The elliptic $K3$-surface~\eqref{Weqn0}
  admits many different elliptic fibrations. The fibrations we
  initially found on $W(17,1)$ and $W(17,3)$ were different, and it was
  only after we discovered that these surfaces are birational that we
  adjusted the calculations in this section so as to find the same
  elliptic fibration. 
\end{Remark}

\section{Some modular curves}

\begin{table}[ht]
\caption{ Copies of $X_0(m)$ on $Z(17,1)$ and $Z(17,3)$ }
\centering
\begin{tabular}{cll} 
  $m$ & \qquad Formula specifying a curve on (a blow up of)
  $z^2 = F_k(T,x,y)$ \\ \hline
2& $F_1(-2 + \eps,-4 + 8 \eps - 5 \eps^2 + \eps^3 + t \eps^4, 
             4 + O(\eps)) = 2^{18} (8 t+1) \eps^4 + O(\eps^5)$  \\
3& $F_3(-1/2 + \eps, 1/2 - \eps + t \eps^3, 1/4 + O(\eps^2))= -2^{-20} (27 t - 16) \eps^4 + O(\eps^5)$ \\
4& $F_1(1 + \eps,- t \eps^2,-1 + O(\eps))= 2^4 (32 t+1) \eps^2 + O(\eps^3)$ \\
5& $F_3(-2 + \eps, 2 - \eps + t \eps^2, 2 + O(\eps))= -2^{12} 3^4 (t^2 - 11 t - 1) \eps^4 + O(\eps^5)$ \\
6& $F_3(-1 + \eps, 1 - 2 \eps + t \eps^2, 2 \eps + O(\eps^2))= (t^2 - 36 t + 36)\eps^{10}  + O(\eps^{11})$\\
7& $F_3(t\eps^{-2}, t \eps^{-3}, \eps^{-2} + O(\eps^{-1}))= t^{22} (t + 1) (t - 27) \eps^{-48}  + O(\eps^{-47})$\\
8& $F_1(\eps,\eps^2 + \eps^3 + 8 t \eps^4, \eps^2 + 2 \eps^3 + O(\eps^4))= 2^4 (t^2 + 6 t + 1) \eps^{18} + O(\eps^{19})$\\ 
9& $F_1(\eps^{-1}, \eps^{-3} + t\eps^{-1}, O(\eps^{-4}))= (t^2 + 20 t - 8) \eps^{-24} + O(\eps^{-23})$ \\
10 & $F_3((t + 1) \eps^2, t^{-1} (t + 1)^3 \eps^4, t^{-1} (t + 1)^3 (\eps^4 + \eps^5) + O(\eps^6))$ \\ & \qquad
$= t^{-4} (t + 1)^{16} (t^2 + 18 t + 1) \eps^{32} + O(\eps^{33})$ \\
11& $F_3(\eps^2 - \eps^3, \eps^3 + \eps^4 + (t - 1) \eps^5, \eps^3 + 2 \eps^4 + O(\eps^5))$ \\ & \qquad
$= t (t^3 + 20 t^2 + 56 t + 44) \eps^{36} +  O(\eps^{37})$ \\
12& $F_3(\eps^{-1}, -\eps^{-1} - (t + 1), t \eps^{-2} + O(\eps^{-1}))= (t^2 - 14 t + 1) \eps^{-24} + O(\eps^{-23})$ \\
13& $F_1(-t\eps^{-1}, (t + 1)\eps^{-1}, t^{-2}(t + 1)^3 + O(\eps))= t^{16} (t^2 + 12 t - 16) \eps^{-20} + O(\eps^{-19})$ \\
14& $F_3(\eps^{-1}, -t \eps^{-3}, t \eps^{-5} + O(\eps^{-4})) = t^2 (t + 1)^4 (t^4 - 14 t^3 + 19 t^2 - 14 t + 1) \eps^{-30} + O(\eps^{-29})$ \hspace{-5em} \\
15& $F_1(\eps^{-1}, t\eps^{-2}, -t \eps^{-4} + O(\eps^{-3})) = t^2 (t - 1)^2 (t^2 - t - 1) (t^2 + 11 t - 1) \eps^{-24} + O(\eps^{-23})$ \hspace{-5em} \\
16& $F_1(-1 + \eps, t \eps, 1 + O(\eps))= (t^2 - 12 t + 4) \eps^{4} + O(\eps^{5})$ \\
18& $F_1(t, -t, -t (t + 1))= t^{16} (t + 1)^2 (t^2 + 10 t + 1)$ \\
19& $F_1(-1 + 2 \eps, (t + 4) \eps, (t + 4)^2 \eps^2 + O(\eps^3)) =
        -8 (t + 3)(t^3 - 2 t + 2) \eps^{4} + O(\eps^{5})$ \\
20& $F_3(-\eps + t^2 \eps^2, \eps - t \eps^2, \eps - \eps^2 + O(\eps^3))= (t^4 + 8 t^3 - 2 t^2 + 8 t + 1) \eps^{12} + O(\eps^{13})$ \\
21& $F_1(-\eps, t (t + 1) \eps^3, (t + 1)^2 \eps^3 + O(\eps^4))= (t^4 + 6 t^3 - 17 t^2 + 6 t + 1) \eps^{16} + O(\eps^{17})$ \\
24& $F_3(t, 1, 0)= (t + 1)^8 (t^3 + t^2 - 1)^4 (t^4 - 8 t^3 + 2 t^2 + 8 t + 1)$ \\ 
25& $F_1(t, t^2, t^2)= -t^{18} (t + 1)^4 (16 t^2 + 4 t - 1)$ \\
27& $F_3(t, t^2 (t + 1), t^2 (t + 1)^2)=
   -t^{18} (t^2 + 2 t + 2)^4 (t - 1) (11 t^3 + 15 t^2 + 9 t + 1)$ \\
32& $F_1(t, t^2 (t + 1), -t^2 (t^3 + t^2 - 1))$ \\
    & \qquad $= t^{16} (t + 1)^4 (t^2 + t + 1)^4
    (t^4 + 8 t^3 + 12 t^2 + 16 t + 4)$ \\
36& $F_1(t, (t + 1) (t^2 + t + 1), (t + 1)^2 (t^2 + 2 t + 2))$ \\
      & \qquad $= (t + 1)^4 (t^3 + t^2 + 2 t + 1)^2 (t^3 + 2 t^2 + 3 t + 1)^4
	    (4 t^4 + 8 t^3 + 12 t^2 + 8 t + 1)$ \\
49& $F_1(t, t (t^2 - 1), -t (t^2 - 1)^2)=
             t^{20} (t + 1)^4 (t^2 - t - 1)^2 (t^4 + 6 t^3 + 3 t^2 - 18 t - 19)$ 
\\ \hline
\end{tabular} 
\label{table:mc}
\end{table}

Let $m \ge 2$ be an integer coprime to $17$. Then any pair of
$m$-isogenous elliptic curves are $17$-congruent with power $k$, where
$k = 1$ if $m$ is a quadratic residue mod $17$, and $k = 3$
otherwise. There is therefore a copy of the modular curve $X_0(m)$ on
the surface $Z(17,k)$. In Table~\ref{table:mc} we explicitly identify
these curves in all cases where $X_0(m)$ has genus $0$ or $1$.  The
polynomials $F_1$ and $F_3$ are those appearing in the statement of
Theorem~\ref{thm1}, and explicitly recorded in
Appendix~\ref{app:formulae}.

\FloatBarrier

In compiling Table~\ref{table:mc} we used the {\tt SmallModularCurve}
database in Magma \cite{Magma} to check the moduli interpretations.  
For example, the entry with $m = 18$ shows that $Z(17,1)$ contains a
curve isomorphic to $y^2 = t^2 + 10t + 1$. We
parametrise this curve
by putting $t=-T/((T + 2)(T + 3))$, and find, using our expressions
for $j_1 + j_2$ and $j_1 j_2$ as rational functions on $W(17,1)$, that
\[ X^2 -(j_1 + j_2)X + j_1j_2 = \big(X - j_{18}(T)\big)
\big(X - j_{18}(6/T) \big) \] 
\noindent where
\[j_{18}(T) =  \frac{((T+2)^{12} - 8(T+2)^9 + 16(T+2)^3 + 16)^3}
                   {(T + 2)^9 ((T + 2)^3 - 8) ((T + 2)^3 + 1)^2 }
 \]
is the $j$-map on $X_0(18)$.

To find most of these curves it was necessary to blow up the surfaces
in Theorem~\ref{thm1}. In such cases we specify the arguments $T,x,y$
of $F_k$ as power series in $\eps$, given to sufficient precision to
determine a unique solution of~\eqref{Weqn0}.  For example, the entry
with $m=20$ shows that blowing up our model for $Z(17,3)$ above
$(T,x,y) = (0,0,0)$ we found a curve isomorphic to $y^2 = t^4 + 8 t^3
- 2 t^2 + 8 t + 1$. Putting this elliptic curve in Weierstrass form we
find it has Cremona label $20a1$, and in particular is isomorphic to
$X_0(20)$.

\section{Examples and further questions}

We restate Conjecture~\ref{conj1} in the case $p=17$.  As usual we say
a $p$-congruence is trivial if it is explained by an isogeny of degree
coprime to $p$.
\begin{Conjecture} 
\label{conj2}
\begin{enumerate}
\item The only non-trivial pairs of symplectically $17$-congruent
  elliptic curves over $\Q$ are the simultaneous quadratic twists of
  the elliptic curves $E_1$ and $E_2$ 
  (with conductors $279809270$ and $3077901970$) as 
  defined in the introduction.
\item The only non-trivial pairs of anti-symplectically $17$-congruent
  elliptic curves over $\Q$ are the simultaneous quadratic twists of
  the elliptic curves $E_1'$ and $E_2'$ 
  (with conductors $3675$ and $47775$) as 
  defined in the introduction.
\end{enumerate}
\end{Conjecture}

We make a related conjecture.
\begin{Conjecture}
\label{conj3}
 Let $\widetilde{Z}(17,k)$ be the surface
in $\Aff^4$ with equations
\[  y^2 + (T + 1)(T - 2) x y + T^3 y = x^3 - x^2 \]
and $z^2 = F_k(T,x,y)$ where $F_k$ is as recorded in
Appendix~\ref{app:formulae}.
\begin{enumerate}
\item The only $\Q$-points on $\widetilde{Z}(17,1)$
lie above one of the curves 
\begin{align*} (x,y) &= ( 0, -T^3 ), 
    ( -T, -T^2 - T ),
    ( -T, -T ),
    ( T^2, T^2 ), 
    ( T^2, -T^4 + T^2 ), \\
    & \qquad ( T^3 + T^2, T^4 + 2 T^3 + T^2 ),
    ( T^3 - T, T^4 - T^2 - T ),
\end{align*}
or the curve $T = 0$, or one of the points in Table~\ref{table-pts}.
\item The only $\Q$-points on $\widetilde{Z}(17,3)$
lie above one of the curves 
\begin{align*} (x,y) &=  (-T, -T ),
    ( -T, -T^2 - T ), 
    ( T^2, -T^4 + T^2 ), \\ & \qquad ( T^3 + T^2, -T^5 - T^4 + T^2 ),
\end{align*}
or the curve $T = 0$, or one of the points in Table~\ref{table-pts}.
\end{enumerate}
\end{Conjecture}

The points on $\widetilde{Z}(17,k)$ lying above one of the curves
listed in Conjecture~\ref{conj3} do not correspond to non-trivial
pairs of $17$-congruent elliptic curves, either because there is an
$m$-isogeny (with $m=18$ or $25$), or the $j$-maps have a pole,
or the point is spurious since $F_k$ vanishes to even multiplicity.

It should be possible to prove that Conjectures~\ref{conj2}
and~\ref{conj3} are equivalent by computing biregular (not just
birational) models for $Z(17,1)$ and $Z(17,3)$. This belief is based
on the fact that in compiling Table~\ref{table:mc} we carried out some
of the necessary blow ups, but did not find any new examples of
$17$-congruences. Nonetheless we leave the full verification to future
work.

The height of a rational number $x = a/b$ (where $a,b$ are coprime
integers) is $H(x) = \max(|a|,|b|)$.  Our evidence for
Conjecture~\ref{conj3} is that we found no further points with $H(T)
\le 3000$ and $H(x) \le 10000$.  We see little hope of proving this
conjecture using existing methods.  A possibly more tractable problem,
the answer to which would still be interesting, would be to determine
all curves of genus 0 or 1 on these surfaces (either over $\Q$ or over
$\Qbar$).

In Table~\ref{table-pts} we list elliptic curves via their Cremona
labels \cite{Cr}, writing instead the conductor followed by a star for
curves beyond the range of Cremona's tables. The latter convention is
only needed for the first entry, where the relevant elliptic curves
are the ones defined in the introduction. The first column records
whether the congruence is symplectic ($k=1$) or anti-symplectic
($k=3$).  The final column records the degree of the isogeny when the
curves are isogenous. Each pair of elliptic curves we list is only
determined up to simultaneous quadratic twist.

\begin{table}[b]
\caption{Some rational points on $Z(17,1)$ and $Z(17,3)$}
\label{table-pts}
$\begin{array}{|c|ccc|c|c|} \hline
k & T & x & y & \text{17-congruent ell. curves} 
& \text{degree} \\ \hline
1 & 1/3 & -2/75 & -11/125 &   {\tt 279809270*}, {\tt 3077901970*}  & - \\
1 & -3 & -27 & 108 &          {\tt 1849a1}, {\tt 1849a2} & 43 \\ 
1 & -5/6 & -5/24 & 5/16 &     {\tt 4489a1}, {\tt 4489a2} & 67 \\ \hline
3 & 1 & 2 & 4 &                 {\tt 27a2}, {\tt 27a4} & 27 \\
3 & 9/7 & 27/49 & -54/49 &       {\tt 3675b1}, {\tt 47775b1} & - \\
3 & -5/14 & 125/392 & 375/1568 & {\tt 1225h1}, {\tt 1225h2} & 37 \\ 
3 & 11/39 & 1771/6591 & 116380/257049 & {\tt 26569a1}, {\tt 26569a2} & 163 
\\ \hline
\end{array}$ 
\end{table}

\FloatBarrier

\newpage

\appendix

\section{Formulae}
\label{app:formulae}

We record the polynomials $F_1(T,x,y)$ and $F_3(T,x,y)$ in
Theorem~\ref{thm1}. These define the double covers $Z(17,1) \to
W(17,1)$ and $Z(17,3) \to W(17,3)$.

\small

\begin{align*}
 F_1(&T,x,y) = x^{10} - 2 T (T - 1) x^8 y + T (T^3 - 2 T^2 - 11 T + 4) x^9 \\
 & - T^2 (T^4 - 3 T^3 - 3 T^2 + 3 T - 10) x^7 y + T^2 (8 T^4 + 58 T^3
   + 15 T^2 - 64 T + 5) x^8 \\ & - T^3 (8 T^5 + 51 T^4 + 80 T^3 + 51 T^2 -
   2 T - 20) x^6 y - 2 T^4 (16 T^4 - 41 T^3 - 205 T^2 \\ & - 97 T + 71) x^7
   + 2 T^4 (16 T^6 + 49 T^5 - 13 T^4 - 142 T^3 - 113 T^2 - T +
   10) x^5 y \\ & - T^4 (149 T^6 + 576 T^5 + 180 T^4 - 956 T^3 - 579 T^2 +
   148 T + 5) x^6 - 2 T^5 (13 T^7 \\ & - 39 T^6 - 229 T^5 - 202 T^4 +
   126 T^3 + 169 T^2 + 9 T - 5) x^4 y + T^5 (80 T^8 + 318 T^7 \\ & -
   192 T^6 - 1800 T^5 - 1376 T^4 + 819 T^3 + 750 T^2 - 67 T - 4) x^5 -
   T^6 (T + 1) (24 T^7 \\ & + 156 T^6 - 24 T^5 - 558 T^4 - 285 T^3 +
   192 T^2 + 21 T - 2) x^3 y - T^6 (16 T^{10} + 72 T^9 \\ & - 239 T^8 -
   1300 T^7 - 870 T^6 + 1952 T^5 + 2295 T^4 + 18 T^3 - 449 T^2 + 4 T +
   1) x^4 \\ & + T^8 (T + 1)^2 (12 T^6 - 50 T^5 - 226 T^4 + 36 T^3 +
   292 T^2 - 31 T - 6) x^2 y \\ & - T^8 (T + 1) (76 T^8 + 273 T^7 - 275 T^6
   - 1505 T^5 - 631 T^4 + 1016 T^3 + 472 T^2 \\ & - 94 T - 4) x^3 - T^{10} (T
   + 1)^3 (T^5 - 14 T^4 + 55 T^3 + 118 T^2 - 68 T - 4) x y \\ & - T^{10} (T +
   1)^2 (131 T^6 + 328 T^5 - 234 T^4 - 700 T^3 - 18 T^2 + 138 T +
   3) x^2 \\ & - T^{13} (T + 1)^4 (T^2 - 2 T + 28) y - 2 T^{13} (T +
   1)^3 (49 T^3 + 63 T^2 - 63 T - 27) x \\ & - 27 T^{16} (T + 1)^4.
\end{align*}

\begin{align*}
 F_3(&T,x,y) =
   x^{10} - 14 T x^8 y - T (4 T^2 - 71 T - 16) x^9 +
   T^2 (17 T^2 - 89 T - 18) x^7 y \\ & - T^2 (14 T^4 + 288 T^3 + 165 T^2 -
   220 T - 19) x^8 + T^3 (76 T^4 + 480 T^3 + 545 T^2 \\ & - 176 T -
   4) x^6 y + T^3 (94 T^6 + 513 T^5 + 234 T^4 - 1412 T^3 - 732 T^2 +
   242 T + 4) x^7 \\ & - T^5 (163 T^5 + 837 T^4 + 1320 T^3 - 72 T^2 - 898 T
   + 106) x^5 y - T^5 (159 T^7 + 590 T^6 \\ & - 103 T^5 - 3276 T^4 -
   3150 T^3 + 1326 T^2 + 820 T - 112) x^6 + T^6 (80 T^7 + 418 T^6 \\ & +
   501 T^5 - 936 T^4 - 2496 T^3 - 948 T^2 + 390 T - 4) x^4 y +
   T^6 (98 T^9 + 386 T^8 \\ & - 350 T^7 - 3439 T^6 - 3894 T^5 + 3010 T^4 +
   4872 T^3 - 306 T^2 - 284 T + 4) x^5 \\ & + T^8 (4 T^8 + 130 T^7 +
   917 T^6 + 2787 T^5 + 4078 T^4 + 2292 T^3 - 482 T^2 - 480 T \\ & -
   60) x^3 y - T^8 (27 T^{10} + 122 T^9 - 177 T^8 - 1496 T^7 - 987 T^6 +
   5032 T^5 + 8446 T^4 \\ & + 1124 T^3 - 2621 T^2 + 54 T - 61) x^4 -
   T^{10} (10 T^9 + 132 T^8 + 738 T^7 + 2126 T^6 \\
   & + 3179 T^5 + 1902 T^4 -
   718 T^3 - 1376 T^2 - 482 T - 256) x^2 y + T^{10} (T^{10} - 48 T^9
\end{align*}

\begin{align*}
   & - 122 T^8 + 882 T^7 + 4304 T^6 + 6244 T^5 + 973 T^4 - 4506 T^3 -
  1714 T^2 + 428 T \\ & - 242) x^3 - T^{12} (T^{10} - T^9 - 70 T^8 -
  368 T^7 - 814 T^6 - 714 T^5 + 237 T^4 + 963 T^3 \\ & + 800 T^2 + 522
  T + 312) x y - T^{12} (26 T^9 + 283 T^8 + 1018 T^7 + 1256 T^6 - 810
  T^5 \\ & - 3237 T^4 - 1848 T^3 + 648 T^2 + 108 T - 265) x^2 - T^{14}
  (T + 2) (T^8 + 12 T^7 + 44 T^6 \\ & + 74 T^5 + 64 T^4 + 20 T^3 - 43
  T^2 - 92 T - 60) y - 2 T^{14} (10 T^8 + 94 T^7 + 323 T^6 \\ & + 471
  T^5 + 129 T^4 - 367 T^3 - 263 T^2 + 69 T + 36) x + T^{16} (T^8 + 4
  T^7 - 8 T^6 \\ & - 66 T^5 - 120 T^4 - 56 T^3 + 53 T^2 + 36 T - 16).
\end{align*}

\normalsize

We also record the change of basis matrices we used in
Sections~\ref{comp:symp} and~\ref{sec:anti}. In each case the rows of
the matrix give the $B_i'$ in terms of the $B_i$.
  
\tiny

\[ \left( \begin{array}{cccccccccccccc}
  1 & 0 & 0 & 0 & 0 & 0 & 0 & 0 & 0 & 0 & 0 & 0 & 0 & 0 \\
  0 & 1 & 0 & 0 & 0 & 0 & 0 & 0 & 0 & 0 & 0 & 0 & 0 & 0 \\
  0 & 0 & 1 & 0 & 0 & 0 & 0 & 0 & 0 & 0 & 0 & 0 & 0 & 0 \\
 -24 & -12 & 28 & 8 & -4 & 0 & 4 & 4 & 0 & -4 & 4 & 0 & 0 & 0 \\
 -34 & -15 & 34 & 7 & -8 & -3 & 0 & 5 & -2 & -9 & 7 & 0 & -1 & 0 \\
 -480 & -288 & 320 & 160 & -64 & -32 & 192 & 160 & 0 & -96 & 96 & 512 & -96 & 0 \\
 248 & 144 & -96 & -144 & 0 & -16 & -192 & -144 & -32 & 48 & 16 & 448 & -48 & 16 \\
 -24 & -56 & -128 & -40 & 0 & -24 & 32 & 56 & -16 & -24 & -8 & -32 & -40 & 8 \\
 -384 & -144 & 688 & 96 & -64 & 0 & 32 & -48 & -16 & 0 & 128 & -64 & 0 & 8 \\
 328 & 204 & -240 & -76 & 32 & 4 & -144 & -92 & 8 & 36 & -20 & 24 & 20 & -12 \\
 -366 & -180 & 96 & 72 & -48 & 12 & 192 & 168 & -12 & -48 & -84 & 36 & 0 & 36 \\
 0 & 0 & 0 & 0 & 0 & 0 & 0 & 0 & 0 & 0 & 0 & 1 & 0 & 0 \\
 0 & 0 & 0 & 0 & 0 & 0 & 0 & 0 & 0 & 0 & 0 & 0 & 1 & 0 \\
 0 & 0 & 0 & 0 & 0 & 0 & 0 & 0 & 0 & 0 & 0 & 0 & 0 & 1 
\end{array} \right) \]

\[ \left( \begin{array}{cccccccccccc}
 -4 & -12 & 8 & -4 & -8 & -8 & 0 & -4 & 0 & 0 & 0 & 0 \\
 0 & 0 & 0 & 0 & 0 & 0 & 0 & 0 & 0 & 0 & 4 & 0 \\
 -56 & -176 & 148 & -68 & -160 & -140 & 36 & -52 & 24 & 24 & -8 & 4 \\
 -104 & -240 & 156 & -76 & -128 & -164 & 12 & -92 & 8 & 8 & 8 & 4 \\
 -16 & -128 & 124 & -36 & -144 & -92 & 12 & -36 & 8 & 8 & 8 & 12 \\
 -40 & -48 & 48 & -24 & -32 & -40 & 48 & -24 & 32 & 32 & 0 & 16 \\
 -64 & -176 & 132 & -60 & -128 & -148 & 4 & -76 & -8 & -8 & 24 & -12 \\
 -36 & -72 & 96 & -96 & 12 & 24 & -36 & -12 & 0 & 24 & 24 & -12 \\
 -60 & -48 & 48 & -24 & -12 & -48 & 60 & -60 & 48 & 48 & 72 & 24 \\
 36 & 72 & -72 & 24 & 60 & 24 & -60 & 12 & -72 & -72 & 24 & -48 \\
 -72 & -84 & 108 & -48 & 0 & 24 & -12 & -24 & 12 & 24 & 0 & 12 \\
 -208 & 204 & -72 & 28 & 364 & 116 & 288 & -44 & 208 & 224 & 32 & 104  
\end{array} \right) \]


\normalsize

\begin{thebibliography}{MM}

\frenchspacing
\renewcommand{\baselinestretch}{1}

\bibitem[AR]{AR} 
A. Adler and S. Ramanan, 
{\em Moduli of abelian varieties}, 
Lecture Notes in Mathematics, {\bf 1644}, Springer-Verlag, Berlin, 1996. 

\bibitem[BM]{BM}
A.J. Best, B. Matschke,
Elliptic curves with good reduction outside of the first six primes,
preprint, 2020, \url{arXiv:2007.10535 [math.NT]}

\bibitem[B]{Billerey}
N. Billerey,
On some remarkable congruences between two elliptic curves,
preprint, 2016, \url{arXiv:1605.09205 [math.NT]}

\bibitem[BL]{BL}
C. Birkenhake and H. Lange, 
{\em Complex abelian varieties},
Second edition, Springer-Verlag, Berlin, 2004. 

\bibitem[BCP]{Magma}
W. Bosma, J. Cannon and C. Playoust, 
The Magma algebra system I: The user language, {\em J. Symb. Comb.} {\bf{24}}, 
235-265 (1997), 
\url{http://magma.maths.usyd.edu.au/magma/}

\bibitem[BHLS]{BHLS}
R. Br\"oker, E.W. Howe, K.E. Lauter and P. Stevenhagen, 
Genus-2 curves and Jacobians with a given number of points,
{\em LMS J. Comput. Math.} {\bf 18} (2015), no. 1, 170--197.

\bibitem[BD]{BD-(44)-splitting}
N. Bruin and K. Doerksen, 
The arithmetic of genus two curves with (4,4)-split Jacobians,
{\em Canad. J. Math.} {\bf 63} (2011), no. 5, 992--1024.

\bibitem[C]{Cr} 
{J.E. Cremona}, 
{\em Algorithms for modular elliptic curves},
Cambridge University Press, Cambridge, 1997,
\url{http://www.warwick.ac.uk/~masgaj/ftp/data/}

\bibitem[CF]{CF}
J.E. Cremona and N. Freitas, 
Global methods for the symplectic type of congruences between elliptic curves, 
to appear in a {\em Revista Matem\'atica Ibroamericana},
\url{arXiv:1910.12290 [math.NT]}

\bibitem[F1]{7and11}
T.A. Fisher, 
On families of 7- and 11-congruent elliptic curves,
{\em LMS J. Comput. Math.} {\bf{17}} (2014), no. 1, 536--564.

\bibitem[F2]{moduli}
T.A. Fisher,
Explicit moduli spaces for congruences of elliptic curves, 
{\em Math. Z.} {\bf 295} (2020), no. 3-4, 1337--1354.

\bibitem[F3]{congr13}
T.A. Fisher, 
On families of 13-congruent elliptic curves, preprint, 2019,
\url{arXiv:1912.10777 [math.NT]}

\bibitem[F4]{magmafile}
  T.A. Fisher,
  On pairs of 17-congruent elliptic curves,
  {\em Magma files accompany this article},
  \url{https://www.dpmms.cam.ac.uk/~taf1000/papers/congr17.html}

\bibitem[FKr]{FK}
N. Freitas and A. Kraus,
On the symplectic type of isomorphisms of the $p$-torsion of elliptic curves,
to appear in {\em Memoirs of AMS}, \url{arXiv:1607.01218 [math.NT]}

\bibitem[FKa]{FreyKani}
G. Frey and E. Kani, 
Curves of genus 2 covering elliptic curves and an arithmetical application,
{\em Arithmetic algebraic geometry} (Texel, 1989), 
G. van der Geer, F. Oort and J. Steenbrink (eds.), 
Progr. Math., {\bf 89}, Birkh\"auser Boston, Boston, MA, 1991, 153--176.

\bibitem[KS]{KS}
E. Kani and W. Schanz, 
Modular diagonal quotient surfaces,
{\em Math. Z.} {\bf{227}} (1998), no. 2, 337--366.

\bibitem[KO]{KO}
A. Kraus and J. Oesterl\'e, 
Sur une question de B. Mazur,
{\em Math. Ann.} {\bf{293}} (1992), no. 2, 259--275.

\bibitem[Kuh]{Kuhn}
R.M. Kuhn, 
Curves of genus 2 with split Jacobian,
{\em Trans. Amer. Math. Soc.} {\bf 307} (1988), no. 1, 41--49.

\bibitem[Kum]{Kumar} 
A. Kumar, 
Hilbert modular surfaces for square discriminants and elliptic
subfields of genus~2 function fields, 
{\em Res. Math. Sci.} {\bf 2} (2015), Art. 24.

\bibitem[Ma]{M}
B. Mazur,
Rational isogenies of prime degree, {\em Invent. Math.} {\bf 44} (1978), no. 2, 129--162.

\bibitem[Me]{Mestre}
J.-F. Mestre, 
Construction de courbes de genre 2 \`a partir de leurs modules,
{\em Effective methods in algebraic geometry}, T. Mora and C. Traverso (eds.), 
Progr. Math., {\bf 94}, Birkh\"auser Boston, Boston, MA, 1991, 313--334.

\bibitem[Se]{SerreMcGill}
J.-P. Serre, 
Abelian $\ell$-adic representations and elliptic curves,
McGill University lecture notes,
W. A. Benjamin, Inc., New York-Amsterdam, 1968.

\bibitem[Sh]{Shaska}
T. Shaska, 
Curves of genus 2 with $(N,N)$ decomposable Jacobians,
{\em J. Symbolic Comput.} {\bf 31} (2001), no. 5, 603--617.

\bibitem[St]{Stein}
W. Stein, 
{\em Modular forms, a computational approach},
Graduate Studies in Mathematics, {\bf 79}, 
American Mathematical Society, Providence, RI, 2007. 

\bibitem[W]{Wamelen}
P.B. van Wamelen, 
Computing with the analytic Jacobian of a genus 2 curve,
{\em Discovering mathematics with Magma},
W. Bosma and J. Cannon (eds.), Algorithms Comput. Math., 
{\bf 19}, Springer, Berlin, 2006, 117--135.

\end{thebibliography}
\end{document}